\documentclass[letterpaper, 10 pt, conference]{ieeeconf}  

\IEEEoverridecommandlockouts                              
\usepackage{graphics} 

\usepackage{amsmath} 
\usepackage{amssymb}  
\usepackage{amsthm}

\usepackage{amsthm,amssymb}

\usepackage[colorlinks=true]{hyperref}

\usepackage{cite}
\usepackage{graphicx}

\newtheorem{thm}{Theorem}[section]
\newtheorem{assum}[thm]{Assumption}
\newtheorem{lem}[thm]{Lemma}
\newtheorem{cor}[thm]{Corollary}
\numberwithin{equation}{section}

\title{\LARGE \bf
Nested Distributed Gradient Methods with Adaptive Quantized Communication
}

\author{Albert S. Berahas, Charikleia Iakovidou and Ermin Wei
\thanks{This work was supported by the DARPA award HR-001117S0039.}
\thanks{A. S. Berahas is with the Department of Industrial and Systems Engineering, Lehigh University, Bethlehem, PA, USA.  {\tt\small albertberahas@lehigh.edu}}%
\thanks{{C. Iakovidou  and E. Wei are with the Department of Electrical and Computer Engineering, Northwestern University,
        Evanston, IL USA,  
        {\tt\small chariako@u.northwestern.edu, ermin.wei@northwestern.edu}}}%
}

\begin{document}

\maketitle
\thispagestyle{empty}
\pagestyle{empty}

\begin{abstract} In this paper, we consider minimizing a sum of local convex objective functions in a distributed setting, where communication can be costly. We propose and analyze a class of nested distributed gradient methods with adaptive quantized  communication (NEAR-DGD+Q). We show the effect of performing multiple quantized communication steps on the rate of convergence and on the size of the neighborhood of convergence, and prove $R$-Linear convergence to the exact solution with  increasing number of consensus steps and adaptive quantization. We test the performance of the method, as well as some practical variants, on quadratic functions, and show the effects of multiple quantized communication steps in terms of iterations/gradient evaluations, communication and cost.


\end{abstract}

\section{INTRODUCTION}

The focus of this paper is on designing and analyzing distributed optimization algorithms that employ multiple agents ($n>1$) in an undirected connected network with the collective goal of minimizing
\begin{align}		\label{eq:prob}
	\min_{x\in \mathbb{R}^p}\quad h(x) =  \sum_{i=1}^n f_i(x),
\end{align}
where $h: \mathbb{R}^p \rightarrow \mathbb{R}$ is the {\it global objective function}, $f_i: \mathbb{R}^p \rightarrow \mathbb{R}$ for each $i\in \{1,2,...,n \}$ is the {\it local objective function} available only to node (agent) $i$, and $x\in \mathbb{R}^p$ is the decision variable that the agents are optimizing cooperatively. Such problems arise in a plethora of applications such as wireless sensor networks \cite{ling2010decentralized,predd2006distributed}, multi-vehicle and multi-robot networks \cite{cao2013overview,zhou2011multirobot}, smart grids \cite{giannakis2013monitoring,kekatos2013distributed} and machine learning \cite{duchi2012,tsianos2012consensus}, to mention a few. 

In order to optimize \eqref{eq:prob} it is natural to employ a {\it distributed optimization algorithm}, where the agents iteratively perform local {\it computations} based on a local objective function and local {\it communications}, i.e., information exchange with their neighbors in the underlying network. To decouple the computation of individual agents, \eqref{eq:prob} is often reformulated as the  following {\it consensus optimization problem} \cite{bertsekas1989parallel},
\begin{align}		\label{eq:cons_prob}
	\min_{x_i \in \mathbb{R}^p}&\quad \sum_{i=1}^n f_i(x_i)\\
	 \text{s.t.} &\quad  x_i = x_j, \quad \forall i, j \in \mathcal{N}_i, \nonumber
\end{align}
where $x_i \in \mathbb{R}^p$ for each agent $i\in \{1,2,...,n \}$ is a local copy of the decision variable, and  $\mathcal{N}_i$ denotes the set of (one-step) neighbors of the $i^{th}$ agent. The {\it consensus constraint} imposed in problem \eqref{eq:cons_prob} enforces that local copies of neighboring nodes are equal; assuming that the underlying network is connected, the constraint ensures that all local copies are equal and as a result problems \eqref{eq:prob} and \eqref{eq:cons_prob} are equivalent.

For compactness, we express problem \eqref{eq:cons_prob} as
\begin{align}		\label{eq:cons_prob1}
	\min_{x_i \in \mathbb{R}^p}&\quad f(\textbf{x}) = \sum_{i=1}^n f_i(x_i)\\
	\text{s.t.} & \quad (\textbf{W}\otimes I_p)\textbf{x} = \textbf{x}, \nonumber
\end{align}
where $\textbf{x} \in \mathbb{R}^{np}$ is a concatenation of all local $x_i$'s, $\textbf{W} \in \mathbb{R}^{n \times n}$ is a matrix that captures information about the underlying graph, $I_p$ is the identity matrix of dimension $p$, and the operator $\otimes$ denotes the Kronecker product operation, with $\textbf{W}\otimes I_p \in \mathbb{R}^{np \times np}$. Matrix $\textbf{W}$, known as the {\it consensus matrix}, is a symmetric, doubly-stochastic matrix with {diagonal elements} $w_{ii}>0$ and {off-diagonal elements} $w_{ij}>0$ ($i\neq j$) if and only if $i$ and $j$ are neighbors in the underlying communication network. This matrix has the property that $(\textbf{W}\otimes I_p) \textbf{x}=\textbf{x}$ if and only if $x_i=x_j$ for all $i$ and $j$ in the connected network, i.e., problems \eqref{eq:cons_prob} and \eqref{eq:cons_prob1} are equivalent; see \cite{bertsekas1989parallel,nedic2009distributed,balCC} for more details.

Distributed optimization algorithms commonly rely on the assumption that vectors are real-valued. However, in digital systems the communication bandwidth is finite and thus information exchanged between agents needs to be quantized. This limitation can prevent algorithms from converging to the true optimal value \cite{doan_accelerating_2018,doan_distributed_2018,pu_quantization_2017,aysal_distributed_2008,minghui_zhu_convergence_2008,charron-bost_randomization_2018}. Moreover, as the dimension increases, communication between agents becomes the bottleneck for performance, and constraining it is essential for achieving fast convergence \cite{reisizadeh_exact_2018,alistarh2017qsgd,rabbat_quantized_2005}.

As a result, there is an extensive body of work studying the effects of {\it quantized communication} on the convergence of distributed algorithms and designing methods robust to {\it quantization error} \cite{doan_accelerating_2018,aysal_distributed_2008,doan_distributed_2018,minghui_zhu_convergence_2008,pu_quantization_2017,charron-bost_randomization_2018,reisizadeh_exact_2018,alistarh2017qsgd,rabbat_quantized_2005,yuan_distributed_2012,li_distributed_2017,nedic_distributed_2008,nedic_distributed_2009,kashyap_quantized_2006,lee2018finite,yi_quantized_2014}. Common approaches include, but are not limited to, allowing the number of quantization levels to approach infinity \cite{nedic_distributed_2008,yuan_distributed_2012,li_distributed_2017,rabbat_quantized_2005,nedic_distributed_2009}, preserving the statistical properties of vectors with probabilistic quantization \cite{aysal_distributed_2008,yuan_distributed_2012,li_distributed_2017}, using weighted averages of quantized consensus and local information \cite{reisizadeh_exact_2018,doan_distributed_2018,yuan_distributed_2012}, designing custom quantizers \cite{lee2018finite,doan_accelerating_2018,pu_quantization_2017} and employing encoding/decoding schemes to alleviate communication load \cite{alistarh2017qsgd,reisizadeh_exact_2018,yi_quantized_2014}. However, with the exception of \cite{lee2018finite,pu_quantization_2017}, these methodologies are unable to achieve geometric convergence rates. In \cite{pu_quantization_2017}, the authors solve a variation of the consensus problem where the local objective functions depend on both local and neighbor variables. The authors of \cite{lee2018finite} extend \cite{di2016next}; however, while gradient tracking algorithms (e.g., \cite{di2016next}) achieve exact convergence at geometric rates, they are not easily tailored to application-specific conditions, such as costly communication or computation. 

%

In this paper, we investigate a class of first-order primal methods that perform nested communication and computation steps, that are adaptive and in which the communication is quantized.  Our work is closely related to a few lines of research that we delineate below:
\begin{enumerate}
	\itemsep0em
	\item \textbf{distributed first-order primal algorithms} \cite{bertsekas1989parallel,balCC,nedic2009distributed,TsitsiklisThes,yuan2016convergence,AsuChapter,SNVStochasticGradient}: methods that use only gradient information,
	\item \textbf{nested} \cite{balCC,chen2012fast,sayed2013diffusion,MouraFastGrad}: methods that decompose the communication and computation steps,
	\item \textbf{communication efficient} \cite{lee2018finite,reisizadeh_exact_2018,alistarh2017qsgd,rabbat_quantized_2005,doan_accelerating_2018,chow2016expander,lan2017communication, shamir2014communication,tsianos2012communication}: methods that incorporate communication considerations in the design,
	\item \textbf{exact} \cite{balCC,extra, di2016next,nedic2017achieving, qu2017harnessing,li2017decentralized,reisizadeh_exact_2018}: methods that converge to the optimal solution using a fixed step length on strongly convex functions,
	\item \textbf{adaptive} \cite{balCC,chen2012fast,MouraFastGrad}: methods that do not perform a fixed number of communication steps per iteration,
	\item \textbf{quantized} \cite{doan_accelerating_2018,aysal_distributed_2008,doan_distributed_2018,minghui_zhu_convergence_2008,pu_quantization_2017,charron-bost_randomization_2018,reisizadeh_exact_2018,alistarh2017qsgd,rabbat_quantized_2005,yuan_distributed_2012,li_distributed_2017,nedic_distributed_2008,nedic_distributed_2009,kashyap_quantized_2006,lee2018finite,yi_quantized_2014}: methods that exchange only quantized information.
\end{enumerate}


The main innovation of this paper is to extend and generalize the existing analysis for a class of nested gradient-based distributed algorithms to account for quantized communication. More specifically, we focus on variants of the NEAR-DGD method \cite{balCC} and analyze a general algorithm that (potentially) takes both multiple and quantized consensus steps at every iteration. We show the effect (theoretically and empirically) of performing  multiple quantized communication steps on the rate of convergence and the size of the neighborhood. Moreover, we prove $R$-Linear convergence to the exact solution with an increasing number of consensus steps and adaptive quantization using a constant steplength on strongly convex functions.

\medskip
The paper is organized as follows.  In Section \ref{sec:NEAR-DGD} we introduce the NEAR-DGD method, and in Section \ref{sec:NEAR-DGD-Q} we present the NEAR-DGD method with quantized communication. We provide a convergence analysis for the method in Section \ref{sec:conv_analysis}. In Section \ref{sec:num_res} we illustrate the empirical performance of the method, and in Section \ref{sec:fin_rem} we provide some concluding remarks and future work.
We conclude this section with a discussion about quantization.

\subsection{Quantization and Adaptive Quantization}
\label{sec:quant}

Let $b_0>0$ denote the number of transmitted bits through a communication channel. The total number of quantization levels is 
    $B_0 = 2^{b_0}$, 
and the distance between two consecutive quantization levels is
\begin{align*}
    \Delta = l_{i+1}-l_i=\frac{u-l}{B_0-1},
\end{align*}
where $l_{i+1}$ and $l_i$ are two consecutive quantization levels and $[l,u]$ is the quantization interval. 

The error due to quantization can be characterized as follows; for some $z \in [l_i,l_{i+1}]$, the quantization error is bounded by
\begin{align*}
    \|z - \mathcal{Q}[z]\|  \leq l_{i+1}-l_i = \Delta,
\end{align*}
where $\mathcal{Q}[z]$ is the quantized version of $z$. Thus, for $\mathbf{z} \in \mathbb{R}^d$ we have the quantization error $\pmb{\epsilon}$ is bounded by
\begin{align*}
  \|\pmb{\epsilon}\| = \|\mathbf{z} - \mathcal{Q}[\mathbf{z}] \| & = \sqrt{\sum_{i=1}^d (z_i - \mathcal{Q}[z_i])^2} \leq 
  \Delta \sqrt{d}.
\end{align*}

Consider the setting in which quantization is adaptive. Let $b_k+b_0 \geq 1$ (for $k=1,2,...$) denote the number of transmitted bits at the $k$th iteration. Then, the total number of quantization levels is
    $B_k = 2^{b_k}B_0$,
and the distance between quantization levels becomes
\begin{align*}
    \Delta_k = \frac{u-l}{2^{b_k}B_0-1} \leq 0.5^{b_k} \Delta.
\end{align*}
Thus, the upper bound of the error due to quantization at the $k$th iteration is
\begin{align*}
    {\Delta}_k \leq \eta^k{\Delta},\quad \text{where $\eta \in (0,1)$}.
\end{align*}

\section{The NEAR-DGD Method}
\label{sec:NEAR-DGD}

In this section, we review the Nested Exact Alternating Recursions method (NEAR-DGD), proposed in \cite{balCC}, upon which we build our quantized algorithm. In its most general form, the $\tau^{th}$ iterate of the NEAR-DGD method can be expressed as
\begin{align*}
	\textbf{x}_\tau = [\mathcal{W}^{t(\tau)}[\mathcal{T}[\cdots
	\rlap{$\overbrace{\phantom{[\mathcal{W}^{t(k)}[\mathcal{T}[}}^{\textbf{x}_k}$}
	[\mathcal{W}^{t(k)}
	\underbrace{[\mathcal{T}[\mathcal{W}^{t(k-1)}[}_{\textbf{y}_k} \mathcal{T}[\cdots\\
	\quad \cdots[\mathcal{W}^{t(1)}[\mathcal{T}[\textbf{x}_0]]]\cdots]]]]]\cdots]]],
\end{align*}
where $\mathcal{T}[\mathbf{x}] = \mathbf{x} - \alpha \nabla \textbf{f}(\textbf{x})$ is the gradient operator, $\mathcal{W}[\mathbf{x}] = (\textbf{W}\otimes I_p)\mathbf{x}$ is the consensus operator and $\mathcal{W}^{t(k)}[x]$ denotes $t(k)$ nested consensus operations (steps),
\begin{align*}
	\mathcal{W}^{t(k)}[x] = \underbrace{\mathcal{W}[\cdots[\mathcal{W}[\mathcal{W}}_{\text{$t(k)$ operations}}[x]]]\cdots].
\end{align*}
Alternatively, one can view the NEAR-DGD method as a method that produces an intermediate iterate $\textbf{y}_k$ after the gradient step, and the iterate $\textbf{x}_k$ after the consensus steps. The iterates $\textbf{x}_k$ and $\textbf{y}_k$ can be expressed as
\begin{gather*}	
	\textbf{x}_k  = \mathcal{W}^{t(k)}[\mathbf{y}_k] = (\textbf{W}\otimes I_p)^{t(k)}\textbf{y}_k =  \textbf{Z}^{t(k)}\textbf{y}_k \label{eq:tw} \\
	\textbf{y}_{k+1}  =\mathcal{T}[\mathbf{x}_k] =  \textbf{x}_k -   \alpha \nabla \textbf{f}(\textbf{x}_k). 
\end{gather*}
By setting the parameters $t(k)$ appropriately, one can recover all the methods proposed in \cite{balCC}.

\section{The NEAR-DGD Method with Quantized Communication}
\label{sec:NEAR-DGD-Q}

In this section, we introduce the NEAR-DGD method with quantized communication---which we call NEAR-DGD+Q. The $\tau^{th}$ iterate of the method can be expressed as
\begin{align*}
	\textbf{x}_\tau = [\mathcal{W}_{\mathcal{Q}}^{t(\tau)}[\mathcal{T}[\cdots
	\rlap{$\overbrace{\phantom{[\mathcal{W}_{\mathcal{Q}}^{t(k)}[\mathcal{T}[}}^{\textbf{x}_k}$}
	[\mathcal{W}_{\mathcal{Q}}^{t(k)}
	\underbrace{[\mathcal{T}[\mathcal{W}_{\mathcal{Q}}^{t(k-1)}[}_{\textbf{y}_k} \mathcal{T}[\cdots\\
	\quad \cdots[\mathcal{W}_{\mathcal{Q}}^{t(1)}[\mathcal{T}[\textbf{x}_0]]]\cdots]]]]]\cdots]]],
\end{align*}
where $\mathcal{T}[\mathbf{x}]$ is the gradient operator defined in the previous section, $\mathcal{W}_{\mathcal{Q}}[\mathbf{x}]$ is the quantized consensus operator and $\mathcal{W}_{\mathcal{Q}}^{t(k)}[x]$ denotes $t(k)$ nested quantized consensus operations (steps),
\begin{align*}
	\mathcal{W}_{\mathcal{Q}}^{t(k)}[x] =  \underbrace{\mathcal{W}[\mathcal{Q}[\cdots[\mathcal{W}[\mathcal{Q}[\mathcal{W}[\mathcal{Q}}_{\text{$t(k)$ operations}}[x]]]]]\cdots]],
\end{align*}
where $\mathcal{Q}[x]$ is the quantization operator. Similar to the NEAR-DGD method, the NEAR-DGD+Q method can be viewed as a method that produces an intermediate iterate $\textbf{y}_k$ after the gradient step, and the iterate $\textbf{x}_k$ after the quantized consensus steps. The iterates $\textbf{x}_k$ and $\textbf{y}_k$ can be expressed as
\begin{gather}	
	\textbf{x}_k  = \mathcal{W}_{\mathcal{Q}}^{t(k)}[\mathbf{y}_k] \label{eq:commq}\\
	\textbf{y}_{k+1}  =\mathcal{T}[\mathbf{x}_k] =  \textbf{x}_k -   \alpha \nabla \textbf{f}(\textbf{x}_k). \label{eq:gradq}
\end{gather}

Since the NEAR-DGD+Q method has an extra step (quantization) compared to the NEAR-DGD method, one can express the method with an additional variable $\textbf{q}_k^j$, that is produced after each quantization step ($j=1,...,t(k)$). For simplicity, let $t(k)=t$. Given $\mathbf{y}_k$, the step \eqref{eq:commq} can be decomposed as the iterative scheme (for $j=1,...,t$)
\begin{align}
    \textbf{q}_k^j &= \begin{cases}
       \mathcal{Q}[\textbf{y}_{k}], \qquad \quad  &\text{ $j=1$} \\
       \mathcal{Q}[\textbf{x}^{j-1}_{k}], \qquad &\text{ $j=2,...,t$}
        \end{cases}\\
    \textbf{x}_{k}^j &= \textbf{Z} \textbf{q}_{k}^j, \qquad \qquad \ \ \ \ \text{ $j=1,...,t$} \label{eq:x_q}
\end{align}
where {$\textbf{Z}=\textbf{W}\otimes I_p$ and} $\textbf{x}_{k}^j$ denotes the variable after the $j^{th}$ round of quantization and consensus. Note, $\textbf{x}_{k}^t$ denotes the output after $t$ rounds of communications (output of step \eqref{eq:commq}), and is the input for the gradient step \eqref{eq:gradq}. For completeness, using this notation, the gradient step \eqref{eq:gradq} is expressed as
\begin{align}       \label{eq:gradqt}	
	\textbf{y}_{k+1} =  \textbf{x}_k^t -   \alpha \nabla \textbf{f}(\textbf{x}_k^t).
\end{align}
If $\mathcal{Q}[\textbf{x}] = \textbf{x}$ (no error due to quantization), we recover the NEAR-DGD method. {Here we consider the NEAR-DGD+Q method where only the communication steps are quantized. One could design a general variant of this method that quantizes both the iterates and the gradients.}



\section{Convergence Analysis}
\label{sec:conv_analysis}

In this section, we analyze the NEAR-DGD+Q method, and its variants.  We begin by assuming that the algorithm takes a fixed number of consensus ($t$) steps per iteration and that the level of quantization is fixed---NEAR-DGD$^t$+Q. We then generalize the results to the case where the number of communication steps varies at every iteration, and the quantization is adaptive---NEAR-DGD$^+$+Q. For brevity we omit some of the proofs, and refer interested readers to \cite{near_dgd_q_arxiv}. We make the following assumptions that are standard in the distributed optimization literature \cite{balCC,nedic2009distributed,yuan2016convergence}.
\begin{assum}\label{assm:Lip}
	 Each local objective function $f_i$  has $L_i$-Lipschitz continuous gradients. We define $L = max_i L_i$.
\end{assum}
\begin{assum}\label{assm:Strong}
	Each local objective function $f_i$ is $\mu_i$-strongly convex. 
\end{assum}

For notational convenience, we introduce the following quantities that are used in the analysis 
\begin{gather}		
	\bar{x}_k = \frac{1}{n}\sum_{i=1}^n x_{i,k}^t, \quad \bar{y}_k = \frac{1}{n}\sum_{i=1}^n y_{i,k},  \nonumber\\
	 g_k =  \frac{1}{n}\sum_{i=1}^n \nabla f_i(x_{i,k}^t), \quad
	 \bar{g}_k =  \frac{1}{n}\sum_{i=1}^n \nabla f_i(\bar{x}_{k})\label{eq:g_barg},
\end{gather}
where $\bar x_k \in \mathbb{R}^p$ and $\bar y_k \in \mathbb{R}^p$ correspond to the average of local estimates, $g_k \in\mathbb{R}^p$ represents the average of local gradients at the current local estimates and $\bar g_k \in\mathbb{R}^p$ is the average gradient at $\bar x_k$.

We note that the gradient step \eqref{eq:gradq} in the NEAR-DGD+Q method can be viewed as a single gradient iteration at the point $x_{i,k}^t$ on the following unconstrained problem
\begin{align*} 	
    \min_{x_i \in \mathbb{R}^p} \quad\sum_{i=1}^n f_i(x_i).
\end{align*}
Moreover, let 
\begin{align*}
    &\mathcal{Q}[\mathbf{y}_k] = \mathbf{y}_k + \pmb{\epsilon}_k^0,\\
    &\mathcal{Q}[\mathbf{x}_k^{j-1}] = \mathbf{x}_k^{j-1} + \pmb{\epsilon}_k^{j-1},\quad &\text{ $j=2,...,t$}
\end{align*}
where $\pmb{\epsilon}_k^{j} \in \mathbb{R}^{np}$, the quantization error, is a concatenation of local quantization errors $\epsilon_{i,k}^j \in \mathbb{R}^p$ for all nodes ($1 \leq i \leq n$). Using the above, \eqref{eq:x_q} can be expressed as
\begin{equation}
    \mathbf{x}_k^t = \mathbf{Z}^t\mathbf{y}_{k} + \sum_{j=0}^{t-1} \mathbf{Z}^{t-j}\pmb{\epsilon}_k^{j} \label{eq:x_k_dgdt2}.
\end{equation}
Multiplying equations \eqref{eq:gradqt} and \eqref{eq:x_k_dgdt2} by $\frac{1}{n}(1_n1_n^T \otimes I)$, and using the fact that $\mathbf{W}$ is a doubly stochastic matrix, we have
\begin{align}
    \bar{y}_{k+1} = \bar{x}_{k} - \alpha g_k, \qquad 
    \bar{x}_{k} = \bar{y}_{k}  + \sum_{j=0}^{t-1} \bar{\epsilon}_k^j, \label{eq:bar1}
\end{align}
where 
\begin{align*}
    \bar{\epsilon}_k^j = \frac{1}{n} \sum_{i=1}^n \epsilon_{i,k}^j, \quad \text{for all $0 \leq j \leq t-1$ and $k=0,1,...$}.
\end{align*}
Note, the errors due to quantization are bounded above as
\begin{align}   \label{eq:quant_error_up}
    \| \pmb{\epsilon}_k^j \| \leq \sqrt{np} \Delta = \tilde{\Delta}, \quad \| \bar{\epsilon}_k^j \| \leq 
    \sqrt{p} \Delta = \frac{\tilde{\Delta}}{\sqrt{n}},
\end{align}
for all $0 \leq j \leq t-1$ and all $k=0,1,...$.
We use these observations to bound the iterates $\textbf{x}_k$ and $\textbf{y}_k$.

\begin{lem} 		\label{lem:twt_bound_iterates}
	\textbf{(Bounded iterates)} Suppose Assumptions \ref{assm:Lip}-\ref{assm:Strong} hold, and let the steplength satisfy
	$\alpha < \frac{1}{L}.$
	Then, the iterates generated by the NEAR-DGD$^t$+Q method \eqref{eq:commq}-\eqref{eq:gradq} are bounded, namely,
	\begin{align*}
	\| \textbf{x}_k^t \| \leq  D + \left( 1 + \frac{2}{\nu}\right)t \tilde{\Delta}, \quad {\| \textbf{y}_k \| \leq D + \frac{2t\tilde{\Delta}}{\nu}},
	\end{align*}
	where $D = \| \textbf{y}_0 - \textbf{u}^\star\| + \frac{\nu+4}{\nu}\|\textbf{u}^\star \|$, $\textbf{u}^\star = [u_1^\star;u_2^\star;...;u_n^\star] \in \mathbb{R}^{np}$, $u_i^\star = \arg\min_{u_i}f_i(u_i)$, $\nu = 2\alpha \gamma$, $\gamma = \min_i \gamma_i$ and $\gamma_i = \frac{2\mu_i L_i}{\mu_i + L_i}$, for any $i \in \{1,2,...,n\}$.
\end{lem}

\begin{proof} Using standard results for the gradient descent method \cite[Theorem 2.1.15, Chapter 2]{nesterov2013introductory}, and noting that $\alpha < \frac{1}{L} < \frac{2}{\mu_i + L_i}$, which is the necessary condition on the steplength, we have for any $i \in \{1,2,...,n\}$
\begin{align*}
	\| x_{i,k}^t - \alpha \nabla {f}_i({x}_{i,k}^t) - {u_i}^\star\| &\leq \sqrt{1 - 2\alpha \textcolor{black}{\gamma}_i} \|x_{i,k}^t - {u_i}^\star\|.
\end{align*}
From this, we have,
\begin{align}
    \| \textbf{x}_k^t - \alpha \nabla \textbf{f}(\textbf{x}_k^t) - \textbf{u}^\star\| &= \sqrt{\sum_{i=1}^n \| x_{i,k}^t - \alpha \nabla {f}_i({x}_{i,k}^t) - {u_i}^\star\|^2}\nonumber\\
	& \leq \sqrt{\sum_{i=1}^n (1 - 2\alpha \textcolor{black}{\gamma}_i) \|x_{i,k}^t - {u_i}^\star\|^2}\nonumber\\
	& \leq \sqrt{ (1- \nu)} \| \textbf{x}_k^t - \textbf{u}^\star \| \label{eq:dgdt_gradd}.
\end{align}
where the last inequality follows from the definition of $\nu$.

Using the definitions of $\nu$, $\textbf{y}_{k+1}$ and Eqs. \eqref{eq:x_k_dgdt2} and \eqref{eq:dgdt_gradd}, we have
\begin{align*}
	\| \textbf{y}_{k+1} - \textbf{u}^\star \| &= \| \textbf{x}_k^t - \alpha \nabla \textbf{f}(\textbf{x}_k^t) - \textbf{u}^\star \|\\
	& \leq \sqrt{(1-\nu)} \| \textbf{x}_k^t -\textbf{u}^\star\|\\
	& = 
    \sqrt{(1-\nu)} \|\mathbf{Z}^t\mathbf{y}_{k} + \sum_{j=0}^{t-1} \mathbf{Z}^{t-j}\pmb{\epsilon}_k^{j} - \textbf{u}^* \|
\end{align*}
The eigenvalues of matrix $\textbf{Z}^t$ are the same as those of the matrix $\textbf{W}^t$. The spectral properties of $\textbf W$ guarantee that the magnitude of each eigenvalue is upper bounded by 1. Hence $\|\textbf{Z}\|\leq 1$ and $\|I-\textbf{Z}^t\|\leq 2$ for all $t$. Hence, the above relation implies that
\begin{align*} 
    \| \textbf{y}_{k+1} - \textbf{u}^\star \| & \leq \sqrt{1-\nu} \left[ \|\mathbf{y}_k - \mathbf{u}^*\| + \sum_{j=0}^{t-1} \| \pmb{\epsilon}_k^{j}\| + 2\|\mathbf{u}^*\|\right] \\
    & \leq \sqrt{1-\nu} \left[ \|\mathbf{y}_k - \mathbf{u}^*\| + 2\|\mathbf{u}^*\| + t \tilde{\Delta}\right].
\end{align*}

Recursive application of the above relation gives,
\begin{align*}
\|\textbf{y}_{k+1}-\textbf{u}^*\| \leq \|\textbf{y}_0 - \textbf{u}^*\| + \frac{4}{\nu}\|\textbf{u}^*\| + \frac{2t}{\nu}\tilde{\Delta}.
\end{align*}
Thus, we bound the iterate as
\begin{align*}
\|\textbf{y}_{k+1}\| &\leq \| \textbf{y}_0-\textbf{u}^*\| + \frac{\nu+4}{\nu}\|\textbf{u}^*\| + \frac{2t}{\nu}\tilde{\Delta} = D + \frac{2t\tilde{\Delta}}{\nu},
\end{align*}
and
\begin{align*}
    \|\textbf{x}^t_{k+1}\| & = \| \mathbf{Z}^t\mathbf{y}_{k+1} + \sum_{j=0}^{t-1} \mathbf{Z}^{t-j}\pmb{\epsilon}_k^{j}\| \\
    & \leq \|\mathbf{y}_{k+1}\| + t \tilde{\Delta} \leq D + \left( 1 + \frac{2}{\nu}\right)t \tilde{\Delta}.
\end{align*}
\end{proof}

Lemma \ref{lem:twt_bound_iterates} shows that the iterates generated by the NEAR-DGD$^t$+Q method are bounded. Since eigenvalues of $\textbf{Z}^t$ and $I-\textbf{Z}^t$ are bounded above by 1 and 2, for any $t$, respectively, the same analysis can be used to show that the iterates generated by the NEAR-DGD$^+$+Q method are also bounded. Note, that the result of Lemma \ref{lem:twt_bound_iterates} reduces to \cite[Lemma V.2.]{balCC} in the case where there is no quantization error.

\begin{lem} 	\label{lem:twt_bound_dev_mean}
\textbf{(Bounded deviation from mean)} If Assumptions \ref{assm:Lip} \& \ref{assm:Strong} hold. Then, starting from $x_{i,0}= s_0$ or $y_{i,0}= s_0$ ($i \in \{1,2,...,n\}$), the total deviation of each agent's estimate ($x_{i,k}$) from the mean is bounded, namely,
\begin{gather}		\label{eq:twt_lem2_p1}
	\| x_{i,k}^t - \bar{x}_k\| \leq \beta^{{t}} D   +  \left( \frac{2\beta^{{t}}}{\nu} + \frac{\sqrt{n}+1}{\sqrt{n}} \right) t\tilde{\Delta}
\end{gather}
and
\begin{gather}
	\| \nabla f_i(x_{i,k}) - \nabla f_i(\bar{x}_k)\| \leq \beta^{{t}} D L_i   +  \left( \frac{2\beta^{{t}}}{\nu} + \frac{\sqrt{n}+1}{\sqrt{n}} \right) t\tilde{\Delta}L_i
	 \label{eq:twt_lem2_p2}\\
	\| g_k - \bar{g}_k \| \leq \beta^{{t}} DL   +  \left( \frac{2\beta^{{t}}}{\nu} + \frac{\sqrt{n}+1}{\sqrt{n}} \right) t\tilde{\Delta}L	\label{eq:twt_lem2_p3}
\end{gather}
for all $k=1,2,\ldots$ and $i \in \{1,2,...,n\}$. Moreover, the total deviation of the local iterates $y_{i,k}$ is also bounded,
\begin{align}	\label{eq:twt_lem2_p4}
	\| y_{i,k} - \bar{y}_k\| \leq \beta^{{t}} D + 2D  +  \left( \frac{2\beta^{{t}}}{\nu} + \frac{4}{\nu} + 2 \right) t\tilde{\Delta}.
\end{align}
\end{lem}

\begin{proof}
Consider,
\begin{align*}
	\|x_{i,k}^t - \bar{x}_k\| & = \left\|x^t_{i,k} - \bar{y}_{k}  + \sum_{j=0}^{t-1} \bar{\epsilon}_k^j\right\| \leq \left\| x^t_{i,k} - \bar{y}_{k}\right \| + \frac{t \tilde{\Delta}}{\sqrt{n}}\\
	&\leq \left\|\textbf{x}_k^t - \frac{1}{n}\left((1_n1_n^T)\otimes I \right) \textbf{y}_k\right\| + \frac{t \tilde{\Delta}}{\sqrt{n}}\\
	& = \left\|\mathbf{Z}^t\mathbf{y}_{k} + \sum_{j=0}^{t-1} \mathbf{Z}^{t-j}\pmb{\epsilon}_k^{j} - \frac{1}{n}\left((1_n1_n^T)\otimes I \right) \mathbf{y}_k\right\| + \frac{t \tilde{\Delta}}{\sqrt{n}} \\
	&\leq \left\|\left(\textbf{W}^{t}  - \frac{1}{n}\left(1_n1_n^T \right)\otimes I \right)\right\|\|\textbf{y}_k\| + \left( \frac{\sqrt{n}+1}{\sqrt{n}}\right)t\tilde{\Delta}\\
	&\leq \beta^t \|\textbf{y}_k\| +  \frac{\sqrt{n}+1}{\sqrt{n}}t\tilde{\Delta}\\
	&\leq \beta^{{t}}\left( D + \frac{2t\tilde{\Delta}}{\nu}\right)  +  \frac{\sqrt{n}+1}{\sqrt{n}}t\tilde{\Delta},
\end{align*}
where the first equality is due to \eqref{eq:bar1} and the last inequality is due to Lemma \ref{lem:twt_bound_iterates}.

The result \eqref{eq:twt_lem2_p2} is a direct consequence of the \eqref{eq:twt_lem2_p1} and the Lipschitz continuity of individual gradients (Assumption \ref{assm:Lip}). To establish the next result \eqref{eq:twt_lem2_p3}, we have
\begin{align*}
    \|g_k - \bar{g}_k\| &= \left\|\frac{1}{n} \sum_{i=1}^{n} \left(\nabla f_i(x_{i,k}^t) -  \nabla f_i(\bar{x}_k^t)\right) \right\| \\
    & \leq \frac{1}{n} \sum_{i=1}^{n} L_i \|x_{i,k}^t - \bar{x}_k^t\| \\
    & \leq \beta^{{t}} DL   +  \left( \frac{2\beta^{{t}}}{\nu} + \frac{\sqrt{n}+1}{\sqrt{n}} \right) t\tilde{\Delta}L
\end{align*}

Finally, for the local $y_{i,k}$ iterates in  \eqref{eq:twt_lem2_p4}, consider
\begin{align*}
	\|y_{i,k} - \bar{y}_k\| & \leq \| x_{i,k}^t - \bar{y}_k\| + \| y_{i,k} - x_{i,k}^t\|\\
	& \leq \left\|\textbf{x}_k^t - \frac{1}{n}\left((1_n1_n^T)\otimes I \right) \textbf{y}_k\right\| + \| \textbf{y}_k - \textbf{x}_k^t \|\\
	& = \left\|\mathbf{Z}^t\mathbf{y}_{k} + \sum_{j=0}^{t-1} \mathbf{Z}^{t-j}\pmb{\epsilon}_k^{j} - \frac{1}{n}\left((1_n1_n^T)\otimes I \right) \mathbf{y}_k\right\| \\
	& \qquad + \left\| \textbf{y}_k -\mathbf{Z}^t\mathbf{y}_{k} - \sum_{j=0}^{t-1} \mathbf{Z}^{t-j}\pmb{\epsilon}_k^{j}\right\| \\
	& \leq \left\|\left(\textbf{W}^{t}  - \frac{1}{n}\left(1_n1_n^T \right)\otimes I \right)\right\|\|\textbf{y}_k\| + t\tilde{\Delta}\\
	& \qquad + \left\| (I -\mathbf{Z}^t)\right\| \| \textbf{y}_k\| + t\tilde{\Delta}\\
	& \leq (\beta^t + 2) \|\textbf{y}_k\| + 2 t\tilde{\Delta}\\
	& \leq \beta^tD + 2D + \left( \frac{2\beta^t}{\nu} + \frac{4}{\nu} + 2 \right)t\tilde{\Delta} 
\end{align*}
where the equality is due to \eqref{eq:bar1} and the last inequality is due to Lemma \ref{lem:twt_bound_iterates}.
\end{proof}

Lemma \ref{lem:twt_bound_dev_mean} shows that the distance between the local iterates $x_{i,k}$ and $y_{i,k}$ are bounded from their means. As was the case with Lemma \ref{lem:twt_bound_iterates}, the result of Lemma \ref{lem:twt_bound_dev_mean} reduces to \cite[Lemma V.2.]{balCC} in the case where there is no quantization error.

We now investigate the optimization error of the NEAR-DGD$^t$+Q method. To this end, we make use of a slightly modified version of an observation made in \cite[Section V]{balCC} that is due to the doubly-stochastic nature of $\textbf{W}$. Namely,
\begin{align}	\label{eq:twt_errors}
	\bar{y}_{k+1} = \bar{y}_k - \alpha g_k + \sum_{j=0}^{t-1} \bar{\epsilon}_k^j,
\end{align}
can be viewed as an inexact gradient descent step for
\begin{align*} 	
	\min_{x\in \mathbb{R}^p} \bar{f}(x) = \frac{1}{n} \sum_{i=1}^n f_i (x),
\end{align*}
 where $\bar{g}_k$ is the exact gradient. {If Assumptions \ref{assm:Lip} and \ref{assm:Strong} hold, then it can be shown that the function $\bar{f}(x)$ is $\mu_{\bar{f}}$-strongly convex and has $L_{\bar{f}}$-Lipschitz continuos gradients.\footnote{Note, $\mu_{\bar{f}}  = \frac{1}{n} \sum_{i=1}^n \mu_i$, and $L_{\bar{f}} = \frac{1}{n} \sum_{i=1}^n L_i$.}}

We should mention that contrary to the analysis in \cite{balCC}, in this work we consider the error instead of the square of the error, and as such we are able to achieve tighter bounds. 

\begin{thm} \label{thm:twt_bound_dist_min}
	\textbf{(Bounded distance to minimum)} Suppose Assumptions \ref{assm:Lip}-\ref{assm:Strong} hold, and let the steplength satisfy
	$\alpha \leq \min \left \{ {\frac{1}{L}, c_6} \right \}$,
	where 
	$c_6 = \frac{2}{\mu_{\bar{f}} +L_{\bar{f}}}$. Then, 
	the iterates generated by the NEAR-DGD$^t$+Q method \eqref{eq:commq}-\eqref{eq:gradq} satisfy 
\begin{align*}
\begin{split}
		\| \bar{x}_{k} - x^\star \| \leq c_1^{k} \| &\bar{x}_{0} - x^\star\|\\& + \frac{c_3\beta^t}{(1-c_1)}+ \frac{c_4\beta^t t\tilde{\Delta}}{(1-c_1)}+ \frac{c_5t\tilde{\Delta}}{(1-c_1)},
		\end{split}
\end{align*}
where
\begin{gather*}
	{c_1 = \sqrt{1 - \alpha c_2}}, \;\; c_2 = \frac{2\mu_{\bar{f}} L_{\bar{f}}}{\mu_{\bar{f}} + L_{\bar{f}}},\\
	c_3 = \alpha DL, \; \; c_4 = \frac{\alpha 2L}{\nu}, \;\; c_5 = \frac{1}{\sqrt{n}}\left( \alpha L(\sqrt{n}+1)+1\right),
\end{gather*}
$x^\star$ is the optimal solution of \eqref{eq:cons_prob1}, $D$ is defined in Lemma \ref{lem:twt_bound_iterates} and $\tilde{\Delta}$ is given in \eqref{eq:quant_error_up}.
\end{thm}

\begin{proof} Using the definitions of the $\bar{x}_k$ and $g_k$, and \eqref{eq:twt_errors}, we have
\begin{align}		
	\| \bar{x}_{k+1} - x^\star \| &\leq \| \bar{x}_{k} - x^\star - \alpha \bar{g}_k\| + \alpha\| \bar{g}_k - {g}_k \| + \left\| \sum_{j=0}^{t-1} \bar{\epsilon}_k^j\right\|\nonumber\\
	& \leq \| \bar{x}_{k} - x^\star - \alpha \bar{g}_k\| + \alpha\| \bar{g}_k - {g}_k \| + \frac{t \tilde{\Delta}}{\sqrt{n}}
	. \label{eq:twt_thm1}
\end{align}
The result of Lemma \ref{lem:twt_bound_dev_mean} bounds the quantity $\| \bar{g}_k - {g}_k \|$. Consider the first term on the right hand side of \eqref{eq:twt_thm1}, and observe that this is precisely the distance to optimality after performing a single gradient step on the function $\bar{f}$. Therefore, by \cite[Theorem 2.1.15, Chapter 2]{nesterov2013introductory}, we have
\begin{align}		\label{eq:neardgdt_thm_bound_distance}
	\| \bar{x}_{k} - x^\star - \alpha \bar{g}_k\| & \leq \sqrt{1 - \alpha c_2}\|  \bar{x}_{k} - x^\star \|.
\end{align}
Combining \eqref{eq:twt_thm1}, \eqref{eq:neardgdt_thm_bound_distance} and using \eqref{eq:twt_lem2_p3},
\begin{align*}			
	\| \bar{x}_{k+1} - x^\star \| & \leq \sqrt{1 - \alpha c_2}\|  \bar{x}_{k} - x^\star \| +\alpha \beta^t DL  \nonumber\\
	& \qquad   +\alpha \left( \frac{2 \beta^t}{\nu} + \frac{\sqrt{n}+1}{\sqrt{n}}\right)t \tilde{\Delta}L + \frac{t\tilde{\Delta}}{\sqrt{n}}.
\end{align*}
Recursive application of the above, and using the definitions of $c_1$, $c_3$, $c_4$ and $c_5$ yields
\begin{align*}
\begin{split}
	\| \bar{x}_{k} - x^\star \| \leq c_1^{k} \| &\bar{x}_{0} - x^\star\|  \\&+ \frac{c_3\beta^t}{(1-c_1)}+ \frac{c_4\beta^t t\tilde{\Delta}}{(1-c_1)}+ \frac{c_5t\tilde{\Delta}}{(1-c_1)},
	\end{split}
\end{align*}
which concludes the proof.
\end{proof}

Theorem \ref{thm:twt_bound_dist_min} shows that the average of the iterates generated by the NEAR-DGD$^t$+Q method converge to a neighborhood of the optimal solution whose radius is defined by the steplength, the second largest eigenvalue of $\textbf{W}$, the number of consensus steps and the quantization error. We now provide a convergence result for the local agent estimates of the NEAR-DGD$^t$+Q method.

\begin{cor}
\label{cor:twt_bound_dist_min} \textbf{(Local agent convergence)} Suppose Assumptions \ref{assm:Lip}-\ref{assm:Strong} hold, and let the steplength satisfy
	$\alpha \leq \min \left \{ {\frac{1}{L}, c_6} \right \}$.
	Then, 
	for $k=0,1,\dots$
\begin{align*}
	\| x_{i,k}^t - x^\star \| &\leq c_1^{k} \| \bar{x}_{0} - x^\star\| + \left(\frac{c_3}{(1-c_1)} + D \right)\beta^t \\
	& \qquad + \left(\frac{c_4}{(1-c_1)} + \frac{2}{\nu} \right)\beta^t t\tilde{\Delta}\\
	& \qquad + \left(\frac{c_5}{(1-c_1)}+ \frac{\sqrt{n} + 1}{\sqrt{n}} \right)t\tilde{\Delta},\\
	\| y_{i,k} - x^\star \| &\leq c_1^{k} \| \bar{x}_{0} - x^\star\| + \left(\frac{c_3}{(1-c_1)} + D \right)\beta^t \\
	& \qquad + \left(\frac{c_4}{(1-c_1)} + \frac{2}{\nu} \right)\beta^t t\tilde{\Delta}\\
	& \qquad + \left(\frac{c_5}{(1-c_1)} + \frac{4}{\nu} + \frac{2\sqrt{n} + 1}{\sqrt{n}} \right)t\tilde{\Delta} + 2D,
\end{align*}
where $c_1$, $c_3$, $c_4$, $c_5$ and $c_6$ are given in Theorem \ref{thm:twt_bound_dist_min}.
\end{cor}

\begin{proof} Using the results from Lemma \ref{lem:twt_bound_dev_mean} and Theorem \ref{thm:twt_bound_dist_min},
\begin{align*}
	\| x_{i,k}^t - x^\star \| &\leq \| \bar{x}_k - x^\star \| + \| x_{i,k}^t - \bar{x}_k \| \\
		& \leq c_1^{k} \| \bar{x}_{0} - x^\star\| + \frac{c_3\beta^t}{(1-c_1)}+ \frac{c_4\beta^t t\tilde{\Delta}}{(1-c_1)}+ \frac{c_5t\tilde{\Delta}}{(1-c_1)}\\
		& \qquad + \beta^{{t}} D   +  \left( \frac{2\beta^{{t}}}{\nu} + \frac{\sqrt{n} + 1}{\sqrt{n}} \right) t\tilde{\Delta}.
\end{align*}

Following the same approach for the local iterates $y_{i,k}$, we have
\begin{align*}
	\| y_{i,k} - x^\star \| &\leq \| \bar{x}_k - x^\star \| + \| y_{i,k} - \bar{x}_k \| \\
	& =  \| \bar{x}_k - x^\star \| + \left\| y_{i,k} - \bar{y}_{k} - \sum_{j=0}^{t-1} \bar{\epsilon}_k^j \right\| \\
	& \leq c_1^{k} \| \bar{x}_{0} - x^\star\| + \frac{c_3\beta^t}{(1-c_1)}+ \frac{c_4\beta^t t\tilde{\Delta}}{(1-c_1)}+ \frac{c_5t\tilde{\Delta}}{(1-c_1)}\\
	& \qquad + \beta^{{t}} D + 2D  +  \left( \frac{2\beta^{{t}}}{\nu} + \frac{4}{\nu} + 2 \right) t\tilde{\Delta} + \frac{t \tilde{\Delta}}{\sqrt{n}}.
\end{align*}
\end{proof}


The main takeaway of Theorem \ref{thm:twt_bound_dist_min} is that the iterates generated by the NEAR-DGD$^t$+Q method converge at a linear rate to a neighborhood of the optimal solution that depends on the consensus and quantization errors. A natural question to ask is whether there is a way to increase the number of consensus steps and diminish the error due to quantization, at every iteration, in order to eliminate the error terms and converge to the optimal solution of \eqref{eq:cons_prob1}. Before we proceed, we should mention that the results of Lemmas \ref{lem:twt_bound_iterates} and \ref{lem:twt_bound_dev_mean} extend to the case with increasing number of consensus steps $t(k)$ and adaptive quantization, where the quantization error at the $k^{th}$ iteration is given by $\tilde{\Delta}_k$.

\begin{thm} \label{thm:tw+_bound_dist_min}
	\textbf{(Bounded distance to minimum)} Suppose Assumptions \ref{assm:Lip}-\ref{assm:Strong} hold, and let the steplength satisfy
	$\alpha \leq \min \left \{ {\frac{1}{L}, c_6} \right \}$,
	Then, 
	the iterates generated by the NEAR-DGD$^+$+Q method \eqref{eq:commq}-\eqref{eq:gradq} satisfy 
\begin{align*}		
	\| \bar{x}_{k+1} - x^\star \| & \leq c_1\|  \bar{x}_{k} - x^\star \| \\
	& \qquad + c_3 \beta^{t(k)} + c_4 \beta^{t(k)} t(k) \tilde{\Delta}_k + c_5 t(k) \tilde{\Delta}_k,
\end{align*}
where $c_1$, $c_3$, $c_4$, $c_5$ and $c_6$ are given in Theorem \ref{thm:twt_bound_dist_min},
$x^\star$ is the optimal solution of \eqref{eq:cons_prob1}, $D$ is defined in Lemma \ref{lem:twt_bound_iterates} and $\tilde{\Delta}_k$ is an upper bound on the quantization error at the $k^{th}$ iteration. Moreover, for any strictly increasing sequence $\{t(k)\}_k$, with $\lim_{k \rightarrow \infty} t(k) \rightarrow \infty$, and strictly decreasing sequence $\{\tilde{\Delta}_k\}_k$, with $\lim_{k \rightarrow \infty} t(k)\tilde{\Delta}_k \rightarrow 0$ the iterates produced by the NEAR-DGD$^+$+Q algorithm converge to $x^\star$. 
\end{thm}

\begin{proof} The proof of Theorem \ref{thm:tw+_bound_dist_min} is exactly the same as that of Theorem \ref{thm:twt_bound_dist_min}, with the difference that the constant number of consensus steps $t$ is replaced by a varying number of consensus steps $t(k)$ and the fixed upper bound on the consensus error $\tilde{\Delta}$ is replaced by a varying upper bound $\tilde{\Delta}_k$. The convergence result follows from the facts that
\begin{align*} 
    \lim_{k\to\infty}\beta^{t(k)}=0, \; \lim_{k\to\infty}\beta^{t(k)} t(k) \tilde{\Delta}_k=0, \; \lim_{k\to\infty} t(k) \tilde{\Delta}_k=0
\end{align*}
for any increasing sequence $\{t(k)\}$ with $\lim_{k \rightarrow \infty} t(k) \rightarrow \infty$ and decreasing sequence $\{\tilde{\Delta}_k\}_k$ with $\lim_{k \rightarrow \infty} t(k)\tilde{\Delta}_k \rightarrow 0$, and thus the size of the error neighborhood shrinks to 0.
\end{proof}

We now show that for appropriately chosen rates (of diminishing errors), the iterates produced by the NEAR-DGD$^+$+Q algorithm converge at an $R$-Linear rate to $x^\star$.

\begin{thm} \label{thm:tw+_rlinear}
\textbf{(R-Linear convergence of the NEAR-DGD$^+$+Q method)} Suppose Assumptions \ref{assm:Lip} \& \ref{assm:Strong} hold, let the steplength satisfy $\alpha \leq \min \left \{ \frac{1}{L}, c_6 \right \}$
, and let ${{t(k) =k}}$ and $\tilde{\Delta}_k = \eta^k\tilde{\Delta}$ ($0<\eta<1$, $\tilde{\Delta}>0$). Then, the iterates generated by the NEAR-DGD$^+$+Q method \eqref{eq:commq}-\eqref{eq:gradq} converge at an $R$-Linear rate to the solution. Namely,
\begin{align*}
	\| \bar{x}_k - x^\star \| \leq C \rho^k
\end{align*}
for all $k=0,1,2,...$, where
\begin{gather*}
C = \max \left\{ \| \bar{x}_0 -x^\star \|, \frac{8(c_3+c_4\tilde{\Delta}+c_5\tilde{\Delta})}{ {(\alpha c_2)^2}}\right\},\\
 \rho = \max\left\{ \beta,\gamma, {1-\frac{\alpha c_2}{2}}\right\},
\end{gather*}
and $c_2$, $c_3$, $c_4$ and $c_5$ are given in Theorem \ref{thm:tw+_bound_dist_min}.
\end{thm}

\begin{proof} We prove the result by induction. First note that their exists a constant $0<\gamma<1$ such that $k \eta^k \leq \gamma^k$ for all $k$. By the definitions of $C$ and $\rho$ the base case $k=0$ holds. Assume that the result is true for the $k^{th}$ iteration, and consider the $(k+1)^{th}$ iteration. By Theorem \ref{thm:tw+_bound_dist_min}, we have
\begin{align*}
	\| \bar{x}_{k+1} - x^\star \| &\leq c_1 \| \bar{x}_{k} - x^\star \| + c_3\beta^{{k}} + c_4 \beta^k \gamma^k \tilde{\Delta} + c_5 \gamma^k \tilde{\Delta}\\
	& \leq c_1 \left(C\rho^{k}\right) + c_3\beta^{{k}} + c_4 \beta^k \gamma^k \tilde{\Delta} + c_5 \gamma^k \tilde{\Delta}\\
	& =  \left(C\rho^{k}\right) \left[ c_1 + \frac{c_3\beta^{{k}} + c_4 \beta^k \gamma^k \tilde{\Delta} + c_5 \gamma^k \tilde{\Delta}}{(C\rho^{k})}\right]\\
	& \leq \left(C\rho^{k}\right) \left[ c_1 + \frac{c_3 + c_4  \tilde{\Delta} + c_5  \tilde{\Delta}}{C}\right]\\
	& \leq \left(C\rho^{k}\right)^2\left[ \sqrt{1- \alpha c_2} + \frac{(\alpha c_2)^2}{8}\right] \\
	& \leq \left(C\rho^{k}\right)^2\left[ 1- \frac{\alpha c_2}{2}\right] \leq C\rho^{k+1}
\end{align*}
where the second inequality is due to the definition of $\rho$, the third inequality is due to the definitions of $c_1$ and $C$, the fourth inequality is due to $\alpha c_2 <1$, and the last inequality is due to the definition of $\rho$.
\end{proof}

\section{Numerical Results}
\label{sec:num_res}

\setlength\abovedisplayskip{3pt}
\setlength\belowdisplayskip{3pt}

In this section, we present numerical results demonstrating the performance of the NEAR-DGD+Q method in terms of iterations,  number of significant digits transmitted and  cost. We measure the cost as proposed in \cite{balCC}, with the difference that instead of tracking the cost per communication round, we track the cost per significant digit sent. Namely, we define 
\begin{align*}
	{\small
    \text{Cost} = \# \text{Digits Transmitted} \times c_c + \# \text{Computations} \times c_g},
\end{align*}
where $c_c$ and $c_g$ are exogenous application-dependent parameters reflecting the costs of transmitting a significant digit and performing a gradient evaluation, respectively.

We investigated the performance of $2$ different variants of the NEAR-DGD$^+$ method and $4$ different quantization schemes on quadratic functions of the form 
\begin{align*}
		 f(x) = \frac{1}{2} \sum_{i=1}^n x^T A_i x + b_i^Tx,
\end{align*}
where each node $i=\{1,...,n\}$ has local information $A_i \in \mathbb{R}^{p\times p}$ and $b_i \in \mathbb{R}^{p} $. The problem was constructed as described in \cite{mokhtari2017network}; we chose a dimension size $p=10$, the number of nodes was $n=10$ and the condition number was $\kappa = 2$. We considered a $4$-cyclic graph topology (each node is connected to its $4$ immediate neighbors). We define the variants of the NEAR-DGD$^+$ method as NEAR-DGD$^+(a,b,c)$, where $a$ is the number of gradient steps, $b$ is the number of initial consensus steps and $c$ is the number of iterations after which the number of consensus steps is doubled. NEAR-DGD$^+(1,1,k)$ is the NEAR-DGD$^+$ method \cite{balCC}; $1$ gradient step and $k$ consensus steps at the $k^{th}$ iteration. We denote different quantization schemes as $Q(a,b,c)$; $a$ is the initial number of digits transmitted, $b$ is the increase factor and $c$ is the number of iterations after which the number of transmitted digits is increased. Note, $Q(a,-,-)$ denotes the quantization scheme transmitting $a$ significant digits (fixed) at every iteration.


Figures~\ref{fig:error_iter} and~\ref{fig:error_iter_p} illustrate the performance of the NEAR-DGD$^+(1,1,k)$ and NEAR-DGD$^+(1,1,50)$ variants, respectively; we plot relative error ($\|\bar{x}_k - x^\star\|^2/\| x^\star \|^2$) in terms of: $(i)$ iterations, $(ii)$ number of significant digits transmitted, and $(iii)$ cost. The cost parameter $c_g$ was set to $1$; we varied the cost parameter $c_c \in \{ 10^{-4}, 10^4\}$. The step length $\alpha$ was manually tuned for all methods.

\begin{figure}
    \centering
    \includegraphics[width=.22\textwidth]{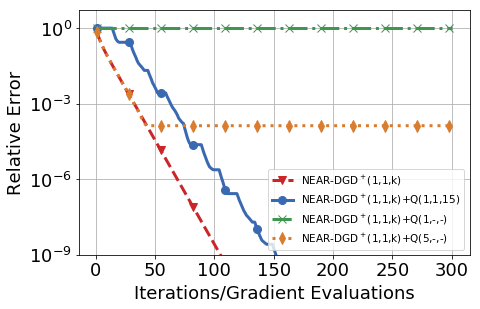}
    \includegraphics[width=.22\textwidth]{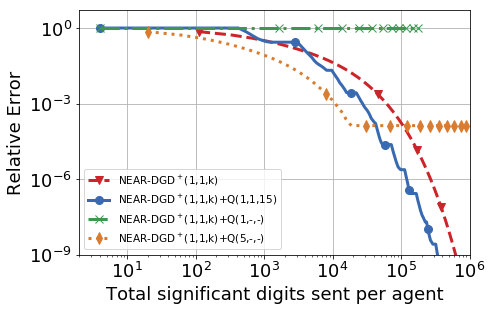}
    
    \includegraphics[width=.22\textwidth]{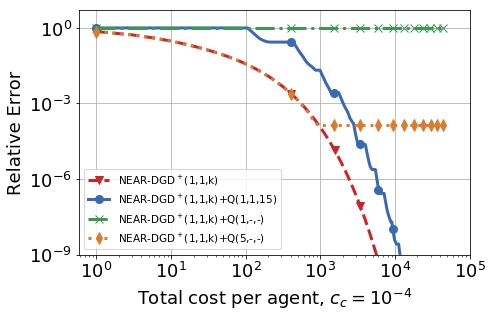}
    \includegraphics[width=.22\textwidth]{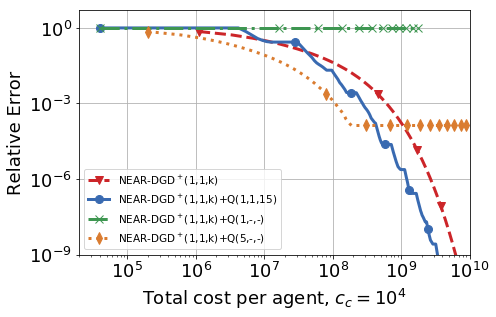}
    \caption{Relative error of NEAR-DGD$^+(1,1,k)$ variants in terms of $(i)$ iterations (top left), $(ii)$ number of significant digits sent (top right),  $(iii)$ cost $c_c=10^{-4}$ (bottom left) $c_c=10^{4}$ (bottom right).}
    \label{fig:error_iter}
\end{figure}
\begin{figure}
    \centering
    \includegraphics[width=.22\textwidth]{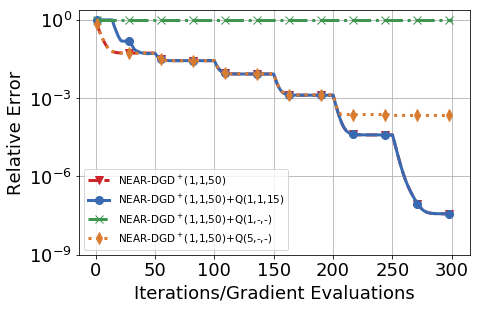}
    \includegraphics[width=.22\textwidth]{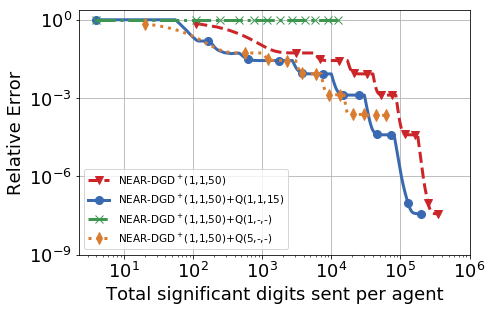}
    
    \includegraphics[width=.22\textwidth]{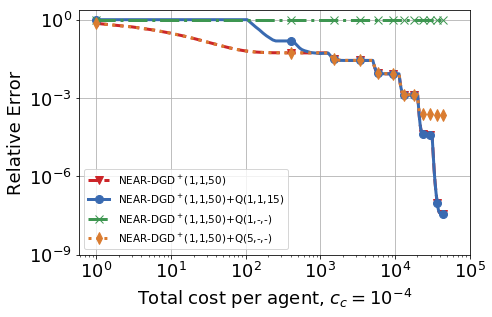}
    \includegraphics[width=.22\textwidth]{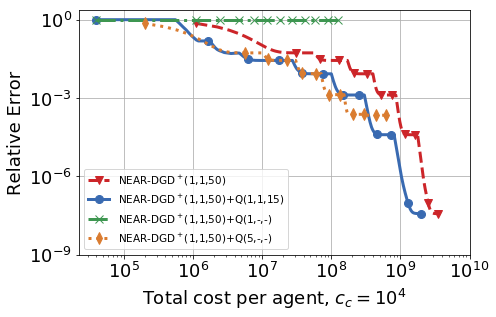}
    \caption{Relative error of NEAR-DGD$^+(1,1,50)$ variants in terms of $(i)$ iterations (top left), $(ii)$ number of significant digits sent (top right),  $(iii)$ cost $c_c=10^{-4}$ (bottom left) $c_c=10^{4}$ (bottom right).}
    \label{fig:error_iter_p}
\end{figure}

As predicted by the theory, when the number of significant digits is fixed the NEAR-DGD$^+$ method only converges to a neighborhood of the solution, the size of which depends on the number of digits transmitted. With regards to  NEAR-DGD$^+(1,1,k)$, although the adaptive quantization variant converges slower than the unquantized variant (in terms of iterations), the quantized variant is able to reach the same accuracy level while transmitting a smaller number of significant digits. On the other hand, for NEAR-DGD$^+(1,1,50)$, the adaptive quantization variant is able to balance the errors and perform equivalently to the unquantized variant while transmitting a smaller number of digits.

In terms of cost, the adaptive quantization variant performs better than the unquantized variant when communication is expensive ($c_c=10^4$). On the other hand, when communication is inexpensive ($c_c=10^{-4}$) and the computation cost dominates, it appears that saving bandwidth with adaptive quantization has little to no benefit. Overall, our experiments indicate that using adaptive quantization one can reduce the communication load without sacrificing accuracy by balancing the errors due to consensus and quantization.


\section{Final Remarks}
\label{sec:fin_rem}

Distributed optimization methods that decouple the communication and computation steps have sound theoretical properties and are efficient over a variety of distributed optimization problems. The NEAR-DGD method is one such method that performs nested communication and gradient steps. In this paper, we generalized the analysis of the NEAR-DGD method to account for quantized communication. Specifically, we showed both theoretically and empirically the effect of performing multiple quantized consensus steps on the rate of convergence and the size of the neighborhood of convergence, and proved $R$-Linear convergence to the exact solution for a method that performs an increasing number of adaptively quantized consensus steps.


\bibliographystyle{ieeetr}
\bibliography{balCC,citationSplitting}

\end{document}